\documentclass[12pt,twoside]{article}
\usepackage{amsmath,amsfonts,amssymb,amsthm}
\usepackage{bbm}
\usepackage{graphicx}

\usepackage{hyperref}
\usepackage{color}
\newcommand{\diff}{\mbox{\rm diff}}
\newcommand{\grad}{\nabla}

\newcommand{\laplace}{\Delta}

\renewcommand{\div}{\grad\cdot}
\newcommand{\N}{\mathbf{N}}
\newcommand{\F}{{\mathcal F}}
\newcommand{\R}{\mathbf{R}}
\renewcommand{\S}{\mathbbm{S}}

\newcommand{\g}{\mbox{\sl g}}

\DeclareMathOperator{\spt}{spt}
\DeclareMathOperator{\Hess}{Hess}
\DeclareMathOperator{\hess}{hess}
\DeclareMathOperator*{\argmin}{arg\,min}

\DeclareMathOperator{\gradient}{grad}
\def\loc{{\mathrm{loc}}}

\newcommand{\Ha}{\ensuremath{\mathcal{H}}}
\newcommand{\D}{\ensuremath{\mathcal{D}}}
\renewcommand{\SS}{\ensuremath{\mathcal{S}}}

\newcommand{\T}{\ensuremath{\mathcal{T}}}
\renewcommand{\H}{H^1_{v_*}}

\newcommand{\M}{\mathcal{M}}

\newcommand{\eps}{\varepsilon}

\newtheorem{theorem}{Theorem}
\newtheorem{prop}[theorem]{Proposition}
\newtheorem{lemma}[theorem]{Lemma}

\newcommand{\tacka}{\, \cdot\,}

\begin{document}

\title{The spectrum of a family of fourth-order equations near the global attractor}

\author{Robert J.\ McCann and Christian Seis}

\maketitle


\begin{abstract}
The thin-film and quantum drift diffusion equations belong to a fourth-order
family of evolution equations proposed in \cite{DenzlerMcCann08} 
to be analogous to the (second-order) porous
medium family. They are $2$-Wasserstein gradient ($W_2$) flows of the generalized Fisher information $I(u)$ just as the porous 
medium family was shown to be the $W_2$ gradient flow of the generalized entropy $E(u)$ by Otto \cite{Otto01}. The identity $I(u) = |\nabla_{W_2} E(u)|^2/2$,  allows us to linearize
the fourth-order dynamics around the self-similar solution, which satisfies 
$\nabla_{W_2} E(u_*)=0$ in self-similar variables. We diagonalize this linearization by coupling the 
resulting relation $\Hess_{W_2} I(u_*) =\Hess_{W_2}^2 E(u_*)$ with the 
diagonalization of the analogous linearization $\Hess_{W_2} E(u_*)$ for
the porous medium flow computed in \cite{Seis13}.  This yields information
about the leading- and higher-order asymptotics  of the equation on $\R^N$
which --- outside of special cases \cite{BernoffWitelski02,CarrilloToscani02,MatthesMcCannSavare09}
--- was inaccessible previously.
\end{abstract}


\section{Introduction}

\subsection{Motivation}
In this manuscript, we investigate the long time behavior of nonnegative solutions to a family of nonlinear fourth-order equations in $\R^N$, namely
\begin{equation}\label{1}
\partial_t u + \div\left(u\grad\left(u^{m-3/2}\laplace u^{m-1/2}\right)\right)\;=\; 0
\end{equation}
with exponent $m \ge 1$. This family has two prominent members which are of physical relevance: If $m=3/2$ the above equation is the thin film equation (with linear mobility), which is a model for the capillarity-driven evolution of a viscous thin film over a solid substrate \cite{OronDavisBankoff97,Myers98} and can also be seen as a lubrication approximation of a Hele--Shaw flow \cite{GiacomelliOtto03}.  In this case, the unknown $u(t,x)$ represents the height of the thin film. The other physically relevant case is $m=1$, where one recovers a simplified quantum drift diffusion model \cite{DLSS91a,DLSS91b,JungelPinnau00}. Here, $u(t,x)$ describes the density of electrons in a semiconductor crystal.

\medskip

Solutions to \eqref{1} are expected to feature different phenomena, depending on the nonlinearity exponent. If $m=1$, the equation is a certain fourth-order analog for the heat equation and thus $u$ should become strictly positive instantaneously. On the other hand, if $m>1$, there are compactly supported solutions (``droplets'') which exhibit a slowly propagating contact line $\partial\spt(u)$. In this case, \eqref{1} turns into a free boundary problem for the support of $u$ and the equation shows some similarities to the porous medium equation. It is this similarity to the heat and porous medium equations --- which is not purely phenomenological --- that we will exploit in this paper to study the long time behavior of solutions.

\medskip

The Cauchy problem for the nonlinear fourth-order equation \eqref{1} is still not completely solved. As the equation is only degenerate parabolic, i.e., it is parabolic only in regions where $u>0$, solutions to \eqref{1} will in general not be classical --- at least when $m>1$.  The theory of existence of weak solutions in the thin film case $m=3/2$ was initiated by {\sc Bernis \& Friedman} \cite{BernisFriedman90} and is by now quite complete, see \cite{BertschDalPassoGarckeGrun98,LidiaGiacomelli04,Grun04} and references therein. On the other hand, regularity and uniqueness of solutions are only partially understood, even in the one-dimensional setting. If $m=1$, existence of nonnegative solutions was established in \cite{JungelPinnau00,GianazzaSavareToscani09}.
The first existence results in the complete range $1\le m\le 3/2$ (this is the range in which the generalized Fisher information, see Eq.\ \eqref{3} below, is convex in the ordinary sense) are relatively recent \cite{MatthesMcCannSavare09}, and they are completely unknown for $m>3/2$. The major difficulty in the analysis of \eqref{1} is the lack of a comparison principle for fourth-order equations. 
Still, the weak solutions which have been constructed are known to preserve nonnegativity.

\medskip

In order to discuss existence theory for \eqref{1} in the case $m>1$ properly, one has to be more precise:  Since \eqref{1} is a free boundary problem of fourth order, three conditions are expected to be needed for well-posedness. In addition to the defining condition $u=0$ on $\partial\spt(u)$, we focus on solutions which preserve mass
\[
\int u(t,x)\, dx\;=\; \int u_0(x)\, dx\;=:\; M\quad\mbox{for any }t>0,
\] 
and satisfy
\begin{eqnarray}
\grad\left(u^{2m-2}\right)\cdot \nu&=&0\quad\mbox{on }\partial\spt(u).\label{4}
\end{eqnarray}
Here, $\nu$ is the outer unit normal vector on $\partial\spt(u)$.
Of course, preservation of mass is also valid for sufficiently regular and rapidly decaying (but presumably positive) solutions in the case $m=1$ because \eqref{1} is a continuity equation on the whole space. In the thin film case $m=3/2$, the condition \eqref{4} on the contact line corresponds to the so-called ``complete wetting'' regime, that is, solutions are assumed to have zero contact angle. Prescribing the contact angle means incorporating Young's law into the equation: At triple junctions $\partial\spt(u)$, the droplets attain instantaneously their energetically most favorable shape. Notice, however, that different choices of contact angles may also be of significance, see, e.g., \cite{Otto98}, but \eqref{4} is apparently the simplest one.

\medskip

As already indicated above, the evolution equation \eqref{1} represents a fourth-order analog for the porous medium equation
\begin{equation}\label{7}
\partial_t u - \laplace u^m\;=\;0.
\end{equation}
This equation is best known for modeling the flow of an isentropic gas through a porous medium if $m>1$ and it becomes the ordinary heat equation if $m=1$. We refer to {\sc V\'azquez}'s monograph \cite{Vazquez07} for a recent survey on the analytical treatment of this equation. In fact, both the fourth order equation \eqref{1} and the porous medium equation \eqref{7} can be interpreted as gradient flows for the Wasserstein distance. (Indeed, {\sc Denzler} and the first author initially {\em introduced} \eqref{1} as a Wasserstein gradient flow \cite[Section 4.3]{DenzlerMcCann08}.) In the case of the fourth order equation \eqref{1}, the dissipating functional is the {\em generalized Fisher information}
\begin{equation}\label{3}
\frac1{2m-1}\int |\grad\left(u^{m-1/2}\right)|^2\, dx,
\end{equation}
while in the case of the porous medium equation, the dissipating functional is the {\em entropy}
\[
\int e(u)\, dx,\quad\mbox{where }e(z)=\left\{\begin{array}{ll}z\ln z&\mbox{ if }m=1,\\ \frac1{m-1}z^m&\mbox{ if }m>1,\end{array}\right.
\]
a fact that was established more than a decade ago in {\sc Otto}'s seminal paper \cite{Otto01}. Also the derivation of \eqref{1} was foreshadowed by the work of Otto who discovered this Wasserstein gradient flow structure in the thin film case $m=3/2$ \cite{Otto98,GiacomelliOtto01}. However, besides this structural similarity there are quite a few close links
between \eqref{1} and \eqref{7}, one of which  is via the following identity: 
\[
\frac{d}{dt} \int e(u(t,\tacka))\, dx\;=\; -\left(\frac{2m}{2m-1}\right)^2\int |\grad\left(u(t,\tacka)^{m-1/2}\right)|^2\, dx
\]
for every smooth solution $u$ of \eqref{7}. That is, the generalized Fisher information encodes the limiting dissipation mechanism for the porous medium flow.  A certain rescaled version of this identity, the entropy-information relation, is at the heart of our analysis, cf.\ Subsection \ref{S:min} and \ref{S:Lin}. The gradient flow structure of \eqref{1} will be reviewed in Subsection \ref{S:GF} below.

\medskip

Both equations \eqref{1} and \eqref{7} allow for a family of self-similar solutions. The former, first discovered by {\sc Smyth \& Hill} \cite{SmythHill88} in the particular case $m=3/2$ and $N=1$ and then generalized in \cite{FerreiraBernis97,DenzlerMcCann08}, is given by
\begin{equation}\label{11}
u_*(t,x)=\frac1{t^{N\alpha }}\hat u_*\left(\frac{|x|}{t^{\alpha}}\right),
\end{equation}
where
\[
e'(\hat u_*(r)) = \left\{
\begin{array}{ll} 
\left(\sigma_M-\gamma r^2\right)&\quad\mbox{if }m=1,\\
\left(\sigma_M - \gamma r^2\right)_+&\quad\mbox{if }m>1.
\end{array}\right.\\
\]
Here, $(\tacka)_+=\max\{\tacka,0\}$, and 
\[
\alpha=\frac1{N(2m-2)+4}\quad\mbox{and}\quad\gamma^2=
\frac{\alpha m^2}{2(2m-1)(N(m-1)+1)}.
\]
Moreover, $\sigma_M$ is a real number (positive if $m>1$) that is determined by the mass constraint
\[
\int u_*\, dx\;=\;M.
\]
As such solutions have a delta measure located at the origin as initial data, self-similar solutions are often called {\em source-type solutions}. For the thin film case it is shown that $u_*$ is the only radially symmetric self-similar solution of \eqref{1} that satisfies the regularity condition \eqref{4}, cf.\ \cite{BernisPeletierWilliams92,FerreiraBernis97}; we expect the same to be true also for the general case $m\ge 1$.
We also remark that $u_*$ resembles the Barenblatt solution (except for the specific values of $\alpha$ and $\gamma$), which is well-known to be the self-similar solution of the porous medium equation \cite{ZeldovicKompaneec50,Barenblatt52,Pattle59}. Moreover, just like the Barenblatt solution describes the long-time behavior of any solution to the porous medium equation \cite{Vazquez03}, $u_*$ is conjectured to be the asymptotic limit of any solution to the fourth order equation. More precisely, for any solution of \eqref{1} with total mass $M$ (and satisfying \eqref{4}) we expect that
\begin{equation}\label{11a}
u(t,\tacka)\approx u_*(t, \tacka)\quad\mbox{as }t\gg0.
\end{equation}
This asymptotic behavior, however, is only partially understood. In the case $m=3/2$, {\sc Carrillo \& Toscani} prove this asymptotics for strong solutions in one space dimension and obtain the sharp rate of convergence. In the range $1\le m\le 3/2$, the long-time behavior is established by {\sc Matthes, Savar\'e} and the first author for JKO-type solutions \cite{JordanKinderlehrerOtto98,AGS} of arbitrary dimension.

\medskip

In the present work, we build upon \eqref{11a} and we go one step further: We compute the complete spectrum and the corresponding eigenfunctions of the displacement Hessian operator which is obtained by linearizing a certain rescaled version of \eqref{1} around its global attractor $u_*$. This spectrum was already (formally) computed by {\sc Bernoff \& Witelski} \cite{BernoffWitelski02} in the case of the one-dimensional thin film equation ($m=3/2$, $N=1$). The knowledge of the spectrum and the corresponding eigenfunctions provides information not only on the slowest rates of convergence that saturate the optimal bounds; it also allows us to extract further information on characteristic geometric pathologies to all orders. In particular, we will make precise statements about the role of solutions generated by translations, shears, and dilations of $u_*$. We obtain the precise information on the spectrum and the corresponding eigenfunctions directly from the spectrum of the displacement Hessian of the linearized (confined) porous medium equation via an explicit entropy-information relation which will be derived in Subsection \ref{S:min} below. Roughly speaking, this entropy-information relation provides us with an explicit formula that associates eigenvalues and eigenfunctions of the porous medium equation with those of the fourth-order equations considered in this paper. The former were computed by the second author in \cite{Seis13} using techniques from mathematical physics developed for the spectral analysis of Schr\"odinger operators. An analogous earlier study for the fast diffusion equation (that is \eqref{7} with $m<1$) is due to {\sc Denzler} and the first author \cite{DenzlerMcCann05}. 

\medskip

We finally caution the reader that even though the analysis presented in the second part of this manuscript, that is the computation of the spectrum and the corresponding eigenfunctions, is rigorous, the linearization that gives rise to the displacement Hessian operator is not. In particular, based on the results presented in Theorem~\ref{T1} below we are rather in the position to {\em conjecture} the higher-order asymptotics of solutions to the nonlinear equation \eqref{1}. An argument for how to close this gap is suggested by the work of {\sc Denzler, Koch}, and the first author \cite{DenzlerKochMcCann13} in the context of the fast diffusion equation: The authors find a rigorously controlled linearization of the fast diffusion equation and a similarity transform that relates the operator which appears in this controlled linearization to the displacement Hessian operator from~\cite{DenzlerMcCann05}. As a result, the authors derive higher-order asymptotics for the fast diffusion equation in weighted H\"older spaces.
\medskip

The remainder of the paper is organized as follows: In Subsection \ref{S:res}, we rescale the equation. Subsection \ref{S:GF} comprises the formal gradient flow structure of the resulting dynamics and in Subsection \ref{S:min} we derive the aforementioned entropy-information relation. In Subsection \ref{S:Lin} we finally linearize the rescaled equation around its global attractor. The second section contains our rigorous work. In Subsection \ref{S:Main}, we present our main results. The self-adjointness of the displacement Hessian operator is established in Subsection \ref{S:SA}. Finally, Subsection \ref{S:proof} contains the proof of our main result. The paper concludes with an appendix on the spectral analysis of the Ornstein--Uhlenbeck operator.

%
%
%
%
%
%
%

\subsection{Rescaling}\label{S:res}

In order to study the long time behavior of solutions to \eqref{1}, it is convenient to rescale equation \eqref{1} in such a way that the global attractor becomes a fixed point of the dynamics. That is, we choose to view the dynamics from perspective receding at a rate 
$| x|\sim t^{\alpha}$ inspired by \eqref{11}. We make the general Ansatz
\[
\hat x = \frac{x}{At^{\alpha}},\quad \hat t = \varphi(t), \quad\mbox{and}\quad v = \frac{t^{N\alpha}}{B} u,
\]
to the effect of
\[
 u( t, x)=\frac{B}{t^{N\alpha}}v \left(\varphi(t),\frac{x}{At^{\alpha}}\right),
\]
where $A,\,B>0$ are two constants and $\varphi: (0,\infty)\to \R$ a smooth increasing function that have to be determined later. Under this transformation, \eqref{1} becomes
\begin{equation}\label{10}
t\varphi'(t)\partial_{\hat t}v - \alpha \hat \grad\cdot \left(\hat xv\right) + \frac{B^{2m-2}}{A^4}\hat \grad\cdot\left(v\hat \grad\left(v^{m-3/2}\hat\laplace v^{m-1/2}\right)\right)\;=\;0.
\end{equation}
To fix the constants $A$ and $B$, we compute how the self-similar solution transforms under the given change of variables. A short computation yields
\[
e'(v_*(\hat r))=\left\{
\begin{array}{ll}
\left(\sigma_M - \gamma A^2\hat r^2\right) - \ln B&\mbox{if }m=1,\\
B^{1-m}\left(\sigma_M - \gamma A^2\hat r^2\right)_+&\mbox{if }m>1,
\end{array}
\right.
\]
%
%
and $v_*$ is independent of the time variable $\hat t$ as desired. For notational convenience, we choose $A$ so that $v_*$ is supported on a ball of radius one (in case $m>1$) and adjust the amplitude of $v_*$ by fixing $B$ so that
\begin{equation}\label{10b}
e'(v_*(\hat r)) = \left\{
\begin{array}{ll}
\frac12\left(1- \hat r^2\right) &\mbox{if }m=1,\\
\frac12\left(1 - \hat r^2\right)_+&\mbox{if }m>1.
\end{array}
\right.
\end{equation}
This is achieved by
\[
A =\left\{
\begin{array}{ll}2^{1/4}&\mbox{if }m=1,\\
 \sqrt{\frac{\sigma_M}{\gamma}} &\mbox{if }m>1,
\end{array}
\right.
\quad\mbox{and}\quad B=\left\{
\begin{array}{ll}
\exp\left(\sigma_M-\frac12\right)&\mbox{if }m=1,\\
(2\sigma_M)^{\frac1{m-1}} &\mbox{if }m>1,
\end{array}
\right.
\]
By a slight abuse of language, we will call $v_*$ the {\em Barenblatt profile}. We remark the above choice of the prefactor in the Barenblatt profile is motivated by the particular formula of the global attractor of the confined porous medium equation considered in \cite{Seis13}, the benefit of which we will become apparent in the subsequent analysis.

\medskip

Substituting these values of $A$ and $B$ into \eqref{10} and choosing $\varphi(t) = \alpha \log t$ gives
\begin{equation}\label{8}
\partial_{\hat t}v -  \hat \grad\cdot \left(\hat xv\right) + \theta\hat \grad\cdot\left(v\grad\left(v^{m-3/2}\hat\laplace v^{m-1/2}\right)\right)\;=\;0,
\end{equation}
where
\begin{equation}\label{8a}
\theta =\frac1\alpha\frac{B^{2m-2}}{A^4} = \frac{2m^2}{(2m-1)(N(m-1)+1)}.
\end{equation}
We call equation \eqref{8} the {\em confined} equation, and the parameter $\theta$ can be interpreted as the relative strength of diffusion compared to confinement.

\medskip

Working with the confined equation \eqref{8} instead of \eqref{1} has the advantage that the global attractor is a stationary solution of the equation. This change of perspective proved to be very useful in the study of long-time asymptotics of both the porous medium equation and the thin film equation \cite{CarrilloToscani00,Otto01,DelPinoDolbeault02,CarrilloToscani02}. Also in regard of the gradient flow formulation of the fourth order equation (on which we shed a light in the subsequent subsection), the above rescaling is beneficial as the Barenblatt profile becomes a ground state of the associated dissipating functional, see the discussion in Subsection \ref{S:min}.
\medskip

We finally remark that under the above rescaling, the initial data of \eqref{1} and \eqref{8} are incompatible. More precisely, the initial datum $v(0,\tacka)$ corresponds to $u(1,\tacka)$. However, as we are interested in the long-time behavior of solutions, this drawback is merely of aesthetic nature.

\medskip

The remainder of this paper is exclusively devoted to the analysis of the confined equation. Hence, to simplify the notation we drop from here on the hats from the time and space variables.

\subsection{Formal gradient flow interpretation}\label{S:GF}
In this subsection, we discuss the formal gradient flow interpretation of \eqref{8}. Even though the following is purely formal, we proceed carefully for the sake of a better understanding of the dynamics. A dynamical system is said to have a gradient flow structure if the evolution follows the steepest descent in an energy landscape, i.e.,
\[
\partial_t v + \gradient I(v)\;=\;0.
\]
The gradient of the energy functional $I$ depends on the Riemannian geometry of the underlying manifold $(\M,\g)$, and is a tangent vector field related to the differential of $I$ via
\begin{equation}\label{6}
\diff I(v).\delta v\;=\; \g_v(\gradient I(v),\delta v)\quad\mbox{for all }\delta v\in \T_v\M,
\end{equation}
where $\T_v\M$ denotes the tangent plane at $v\in\M$. With this, the gradient flow can be rewritten in a gradient-free way, namely
\[
\g_v(\partial_t v, \delta v) +\diff I(v).\delta v\;=\;0 \quad\mbox{for all }\delta v\in \T_v\M,
\]
or variationally,
\begin{equation}\label{2}
\partial_t v= \argmin_{\delta v\in \T_v\M}\left\{ \frac12\g_v(\delta v, \delta v) + \diff I(v).\delta v\right\}.
\end{equation}

\medskip

Interpreting the fourth order equation \eqref{8} (with \eqref{4}) as a gradient flow, the energy has to be chosen as the modified generalized Fisher information
\[
I_{\theta}(v)=\frac{\theta}{2m-1}\int |\grad\left(v^{m-1/2}\right)|^2\, dx + \frac{1}2 \int |x|^2v\, dx.
\]
The manifold is the set of all nonnegative functions on $\R^N$ with fixed total mass $M$
\[
\M:= \left\{v:\R^N\to [0,\infty):\: \int v\, dx =M\right\}.
\]
In general, tangent fields at $v\in\M$ must be mass preserving and nonnegative where $v=0$. For tangent fields satisfying $\spt(\delta v)\subset \spt(v)$, 
the metric tensor is a weighted $H^{-1}$ inner product defined by
\begin{equation}\label{2b}
\g_v(\delta v, \delta v) = \int v |\grad \psi|^2\, dx.
\end{equation}
Notice that the condition on $\spt(\delta v)$ is trivially satisfied if $m=1$ since in this case solutions are expected to be positive everywhere in $\R^N$. In that situation, $\psi$ is related to $\delta v$ via the elliptic problem
\begin{equation}\label{2a}
-\div\left(v\grad \psi\right)\;=\;\delta v\quad\mbox{in }\R^N.
\end{equation}
For $m>1$, however, compactly supported solutions are known to exists and thus $\spt(\delta v)\subset \spt(v)$ is not always true. In this case, $\psi$ and $\delta v$ are related via the elliptic boundary value problem
\begin{equation}\label{5}
\begin{array}{rcll}-\div\left(v\grad \psi\right)&=&\delta v&\quad\mbox{in }\spt(v),\\ v\grad \psi\cdot \nu&=&0&\quad\mbox{on }\partial \spt(v).\end{array}
\end{equation}
The boundary conditions on $\psi$ have to be understood asymptotically, cf.\ \eqref{9b} below. For tangent fields with $ \spt(\delta v)\not\subset \spt(v)$, we simply set $\g_v(\delta v,\delta v)=\infty$. Notice that we can define the metric tensor also variationally (and simultaneously for all $m$) by
\[
\frac12\g_v(\delta v,\delta v) = \sup_{\varphi}\left\{ -\frac12\int v|\grad\varphi|^2\, dx + \int \varphi \delta v\, dx\right\},
\]
where the supremum is taken over all smooth functions $\varphi$ on $\R^N$. Indeed, the right hand side is finite only if $\spt(\delta v)\subset \spt(v)$ and in this case, the maximizer $\psi$ satisfies either \eqref{2a} or \eqref{5} depending on whether $v$ is strictly positive or not, and then \eqref{2b} follows via an integration by parts.

\medskip

Notice that Otto, who first discovered Wasserstein gradient flows, argued even more formally when requiring $-\div\left(v\grad\psi\right)=\delta v$ in $\R^N$, cf.\ \cite[Eq.\ (8)]{Otto01}. However,  due to the (eventually) singular behavior of solutions at $\partial\spt(u)$, the latter equation should rather be understood distributionally. Then, the asymptotic boundary conditions in \eqref{5} arise as natural boundary conditions associated with distributional solutions, see also \cite[Subsection 1.2]{Seis13}.

\medskip

In the following we show that the gradient flow associated with these ingredients gives indeed rise to our family of fourth-order equations. We only give the argument in the case $m>1$; the case $m=1$ is very similar. For this purpose, we first compute
\[
\diff I_{\theta}(v).\delta v\;=\; \theta\int \grad\left(v^{m-1/2}\right)\cdot \grad \left(v^{m-3/2}\delta v\right)\, dx +\frac{1}2 \int |x|^2\delta v\, dx.
\]
Then, invoking \eqref{5}, the gradient flow \eqref{2} becomes $\partial_t v + \div\left(v\grad \psi \right)=0$, where $\psi$ is the solution of
\begin{eqnarray*}
\psi&=&\argmin_{\varphi}\left\{\frac12\int v|\grad \varphi|^2\, dx-\theta\int \grad\left(v^{m-1/2}\right)\cdot \grad \left(v^{m-3/2}\div(v\grad\varphi)\right)\, dx\right.\\
&&\hspace{4em}\left.  -\frac{1}2 \int |x|^2\,\div(v\grad\varphi) dx\right\}.
\end{eqnarray*}
Here, the minimum is taken over all $\varphi$ satisfying $v\grad\varphi\cdot\nu=0$ on $\partial \spt(v)$. Solving this convex minimization problem is easy. We compute the Euler--Lagrange equations
\begin{eqnarray*}
\lefteqn{\int v\grad\psi\cdot\grad\varphi\, dx}\\
 &=&\theta\int \grad\left(v^{m-1/2}\right)\cdot \grad \left(v^{m-3/2}\div(v\grad\varphi)\right)\, dx
 +\frac{1}2 \int |x|^2\,\div(v\grad\varphi) dx ,
\end{eqnarray*}
for all $\varphi$ satisfying $v\grad\varphi\cdot\nu=0$ on $\partial \spt(v)$. Integration by parts and using the asymptotic boundary conditions for $\varphi$ yields
\begin{eqnarray*}
\lefteqn{\int v\grad\psi\cdot\grad\varphi\, dx}\\
&=& \frac{\theta(m-1/2)}{2m-2} \int_{\partial\spt(v)} \grad\left(v^{2m-2}\right)\cdot \nu \div\left(v\grad\varphi\right)\, dx\\
&&\mbox{} + \int v\grad\left(\theta v^{m-3/2}\laplace v^{m-1/2} - \frac12|x|^2\right)\cdot\grad\varphi\, dx.
\end{eqnarray*}
It remains to apply the fundamental lemma of calculus of variations to deduce that $\grad\psi=\theta\grad\left( v^{m-3/2}\laplace v^{m-1/2}\right) - x$, which yields \eqref{8}, and $\grad\left(v^{2m-2}\right)\cdot\nu=0$ on $\partial\spt(v)$, i.e., \eqref{4}.


\medskip

For the sake of completeness and to explain why the above gradient flow is referred to as a Wasserstein gradient flow, we finally introduce the {\em Wasserstein distance}. Given two nonnegative functions $w_0$ and $w_1$ with same total mass, the Wasserstein distance between $w_1$ and $w_0$ is defined as
\[
d_2(w_1,w_0)^2:= \inf_{\pi\in\Gamma(w_1,w_0)}\iint |x-y|^2\, d\pi(x,y),
\]
where the set $\Gamma(w_1,w_0)$ consists of all Borel measures $\pi\ge0$ on the product space $\R^N\times \R^N$ with marginals $w_1$ and $w_0$, i.e.,
\[
\iint \zeta(x)\, d\pi(x,y)\;=\; \int \zeta w_1\, dx\,\quad\mbox{and}\quad\iint \zeta(y)\, d\pi(x,y)\;=\; \int \zeta w_0\, dy,
\]
for all bounded continuous functions $\zeta$ on $\R^N$. Informally, the Wasserstein distance measures the optimal transport cost necessary to  transfer mass from the source $w_0$ to the sink $w_1$. In \cite{BenamouBrenier00}, {\sc Benamou \& Brenier} established a relation between this mass transport problem and fluid mechanics by observing that the geodesics distance induced by the metric tensor $\g$ on $\M$ is given by the Wasserstein distance:
\begin{eqnarray*}
\lefteqn{d_2(w_1,w_0)^2}\\
&=&\inf\left\{\int_0^1\g_w(\partial_s w(s),\partial_s w(s))\, ds:\: w:[0,1]\to\M, w(0)=w_0,\,w(1)=w_1\right\}.
\end{eqnarray*}
The curves $w(s)$ attaining the infimum of 
this variational principle were introduced by one of us earlier \cite{McCann97}. For an introduction to optimal transport, we refer to {\sc Villani}'s monograph \cite{Villani03}.

\subsection{An entropy-information relation}\label{S:min}

The aim of this subsection is the (formal) derivation of the entropy-information relation
\begin{equation}\label{9}
(N(m-1)+1)I_{\theta}(v)- \frac12\g_v(\gradient E(v), \gradient E(v)) =\left\{\begin{array}{ll}NM&\mbox{if }m=1,\\ N(m-1) E(v)&\mbox{if }m>1,\end{array}\right.
\end{equation}
where $E$ is the generalized entropy
\[
E(v)=\int e(v)\, dx + \frac12 \int |x|^2v\, dx,
\]
which is the dissipating functional in the Wasserstein gradient flow interpretation of the confined porous medium equation
\begin{equation}\label{5c}
\partial_t v-\div\left(xv\right) - \laplace\left(v^m\right)\;=\;0,
\end{equation}
cf.\ {\sc Otto} \cite{Otto01}. 
The entropy-information relation was observed by {\sc Matthes, Sava\-r\'e}, and the first author \cite[Corollary 3.2]{MatthesMcCannSavare09} and builds upon an observation of {\sc Carrillo \& Toscani} in \cite{CarrilloToscani00}. The above formula will turn out to be at the heart of the subsequent analysis. Indeed, \eqref{9} establishes a link between the generalized entropy and the generalized Fisher information, that is, between the dissipating functionals in the Wasserstein gradient flow formulation of the confined porous  medium equation and the confined fourth order equation, respectively. As a consequence, linearizing \eqref{8} around the global attractor $v_*$ yields an explicit formula for the fourth order displacement Hessian in terms of the porous medium displacement Hessian computed and studied by the second author in \cite{Seis13}. Moreover, eigenvalues and eigenfunctions of the porous medium displacement Hessian will immediately translate into those of the fourth order displacement Hessian.

\medskip

The second term in \eqref{9} takes the form
\begin{equation}\label{5d}
\g_v(\gradient E(v), \gradient E(v))\;=\; \int v|\grad\left( e'(v)+\frac12|x|^2\right)|^2\, dx.
\end{equation}
Before confirming \eqref{5d} and deriving \eqref{9}, we like to point out two immediate consequences of these two formulas: First, \eqref{5d} implies that, for any given mass $M$,  $v_*$ is the unique critical point of the generalized entropy $E$, and thus, by the strict convexity of this functional, $v_*$ is the unique ground state. Thanks to \eqref{9}, it then follows that the same property is passed along to the confined fourth order equation: $v_*$ is also the unique ground state of the generalized Fisher information $I_{\theta}$.

\medskip 

We start with verifying \eqref{5d}. The gradient on a Riemannian manifold can be computed via the differential, cf.\ \eqref{6}. We have
\[
\g_v(\delta v, \gradient E(v))\;=\; \diff E(v).\delta v\;=\; \int\left(e'(v) + \frac12 |x|^2\right)\delta v\, dx,
\]
for any tangent field $\delta v$. Letting $\psi$ be the function related to $\delta v$ via \eqref{5}, an integration by parts yields
\[ 
\g_v(\delta v, \gradient E(v))\;=\; \int v\grad \left(e'(v) + \frac12 |x|^2\right)\cdot \grad\psi\, dx,
\]
and thus \eqref{5d} follows upon choosing $\delta v=\gradient E(v)$.




\medskip

The entropy-information relation \eqref{9} is essentially a consequence of \eqref{5d} by expanding the square in the integrand and regrouping terms. Indeed, we have
\begin{eqnarray*}
\lefteqn{\int v|\grad\left( e'(v)+\frac12|x|^2\right)|^2\, dx}\\
&=& 
\int v|\grad \left(e'(v)\right)|^2\, dx + \int |x|^2 v\, dx + 2\int x\cdot v\grad\left(e'(v)\right)\, dx\\
 &=&\left(\frac{m}{m-1/2}\right)^2\int |\grad\left(v^{m-1/2}\right)|^2\, dx + \int |x|^2 v\, dx - 2N\int v^m\, dx.
\end{eqnarray*}
If $m=1$, then $\theta=2$ and thus \eqref{9} follows immediately via \eqref{5d}. If $m>1$, we rewrite with the help of $\theta$, cf.\ \eqref{8a},
\begin{eqnarray*}
\lefteqn{\frac12\int v|\grad\left( e'(v)+\frac12|x|^2\right)|^2\, dx}\\
&=&(N(m-1)+1) \left(\frac{\theta}{2m-1} \int |\grad\left(v^{m-1/2}\right)|^2\, dx+ \frac12\int |x|^2 v\, dx\right)\\
&&\mbox{} - N(m-1)\left(\int e(v)\, dx + \frac12 \int |x|^2v\, dx\right).
\end{eqnarray*}
By the definition of $I_{\theta}$ and $E$ this immediately yields \eqref{9}.

\subsection{Linearization}\label{S:Lin}

In this subsection, we finally linearize the confined fourth order equation around its global attractor $v_*$. However, instead of linearizing \eqref{8} directly, we compute the Hessian of the generalized Fisher information $I_{\theta}$ with respect to the Riemannian geometry introduced in Subsection \ref{S:GF} above. Although both approaches are formally equivalent, the latter has the advantage that a class of tangent vectors are intrinsically given, and thus the linearized operator is naturally equipped with an analytical framework. The same strategy was already pursued in the analogous studies \cite{DenzlerMcCann05} and \cite{Seis13} in the context of the fast diffusion equation and porous medium equation, respectively.

\medskip

We first review some general facts from Riemannian geometry. Recall that for any function $F$ on $(\M,\g)$ with $\gradient F(v_*)=0$ for some $v_*\in\M$, the Hessian at $v_*$ can be computed by taking the second derivatives along any curve through $ v_*$. Indeed, if $(r,s)\mapsto v_{(r,s)}\in \M$ is a family of curves satisfying  $v_{(0,0)}=v_*$ and if $\delta v=\left.\partial_r\right|_{r,s=0}v_{(r,s)}$ and $\delta w=\left.\partial_s\right|_{r,s=0}v_{(r,s)}$, then
\begin{eqnarray*}
\Hess F(v_*)(\delta v,\delta w)&:=& \left.\frac{d^2}{drds}\right|_{r,s=0}F(v_{(r,s)})\\
 &=&\left. \frac{d}{dr}\right|_{r,s=0} \g_{v_{(r,s)}}\left(\partial_s v_{(r,s)},\gradient F(v_{(r,s)})\right)\\
&=& \g_{v_*} \left(\delta w,\hess F(v_*)\delta v\right) + \g_{v_*}\left( \left.D_r\partial_s\right|_{r,s=0} v_{(r,s)},\gradient F(v_*)\right)\\
&=&\g_{v_*} \left( \delta w,\hess F(v_*)\delta v\right),
\end{eqnarray*}
where $D_r$ denotes the covariant derivative along the curve $r\mapsto v_{(r,s)}$, cf.\ \cite[Section 2.1.2]{Peterson}, and we have used criticality of $v_*$ in the last identity. Likewise, if $s\mapsto v_{s}\in \M$ is a curve satisfying  $v_0=v_*$ and if $\delta v=\left.\partial_s\right|_{s=0}v_{s}$, then
\begin{eqnarray*}
\lefteqn{\left.\frac{d^2}{ds^2}\right|_{s=0} \g_{v_s}(\gradient F(v_s), \gradient F(v_s))}\\
&=& 2\left.\frac{d}{ds}\right|_{s=0} \g_{v_s}\left(\gradient F(v_s), \hess F(v_s)\partial_sv_s\right)\\
&=& 2 \g_{v_*}(\hess F(v_*)\delta v,\hess F(v_*)\delta v) + 2\g_{v_*}\left(\gradient F(v_*), \left.D_s\right|_{s=0}\left(\hess F(v_s)\partial_s v_s\right)\right)\\
&=& 2 \g_{v_*}(\hess F(v_*)\delta v,\hess F(v_*)\delta v).
\end{eqnarray*}

\medskip

As a consequence, in our particular setting,
since $v_*$ is a common minimizer of $I_{\theta}$ and $E$, differentiating \eqref{9} along curves $s\mapsto v_s\in\M$ with $v_0=v_*$ and $\delta v=\left.\partial_s\right|_{s=0}v_s$ yields
\begin{eqnarray}
\lefteqn{(N(m-1)+1)  \Hess I_{\theta}(v_*)(\delta v,\delta v)}\nonumber\\
&=& \g_{v_*}(\hess E(v_*)\delta v, \hess  E(v_*)\delta v) + N(m-1) \g_{v_*}(\delta v, \hess E(v_*)\delta v).\label{9a}
\end{eqnarray}
(This formula is also consistent with \eqref{9} for $m=1$, since in this case the constant $M$ vanishes in the linearization.) From here it is easy to see the relationship \eqref{9aa} of the displacement Hessian $\Ha_I$ we seek to the operator studied in \cite{Seis13}; the action of $\Ha_I$ on $\H$ can also be deduced from the explicit computation of $\Ha_E$ given there. However, for the convenience of the reader, let us also calculate the terms on the right-hand side directly. We first recall that
\[
\g_v( \delta w,\gradient E(v))\;=\; \int \left(e'(v) + \frac12|x|^2\right)\delta w\, dx,
\]
and thus the Hessian of the entropy functional is
\[
\g_{v_*}(\delta w,\hess E(v_*)\delta v)\;=\; \int e''(v)\delta v\delta w\, dx\;=\; \int mv_*^{m-2}\delta v\delta w\, dx.
\]
Letting $\psi$ and $\varphi$ be the functions related to $\delta v$ and $\delta w$ via \eqref{5} and integrating by parts, we rewrite the above identity as	
\[
\g_{v_*}(\delta w,\hess E(v_*)\delta v)\;=\; \int v_*\grad\left(-m v_*^{m-2} \div\left(v_*\grad \psi\right)\right)\cdot \grad\varphi\, dx.
\]
We have thus derived the Hessian of the generalized entropy. Following  \cite{DenzlerMcCann05,Seis13}, we now introduce the {\em displacement Hessian} $\Ha_E$ associated with the porous medium flow, namely
\[
\Ha_E\psi:=  -m v_*^{m-2} \div\left(v_*\grad\psi\right).
\]
A short computation yields
\[
\Ha_E\psi(x)\;=\; \left\{\begin{array}{ll} -\laplace \psi(x) + x\cdot \grad\psi(x) &\mbox{if }m=1,\vspace{0.3em}\\-\frac{m-1}2(1-|x|^2)\laplace \psi(x) + x\cdot \grad\psi(x) &\mbox{if }m>1.\end{array}\right.
\]
Notice that in the case $m=1$, the operator $\Ha_E$ is the well-known Ornstein--Uhlenbeck operator.

\medskip

In particular, choosing $\delta v=\delta w$, we obtain
\[
\g_{v_*}(\delta v, \hess E(v_*)\delta v)\;=\; \int v_* \grad\Ha_E\psi\cdot\grad\psi\, dx,
\]
and with $\delta w=\hess E(v_*)\delta v$, we have
\[
\g_{v_*}(\hess E(v_*)\delta v, \hess  E(v_*)\delta v) \;=\; \int v_* |\grad \Ha_E\psi|^2\, dx\;=\; \int v_* \grad \Ha^2_E\psi\cdot \grad \psi\, dx,
\]
where in the last equality we integrated by parts and used the symmetry of $\Ha_E$. Substituting these identities into \eqref{9a}, we then conclude
\[
(N(m-1)+1)\Hess I_{\theta}(v)(\delta v, \delta v)\;=\;\int v_* \grad\psi\cdot\grad\left( \Ha_E^2 +N(m-1)\Ha_E\right)\psi\, dx.
\]
The {\em displacement Hessian $\Ha_I$} associated with the confined fourth order equation \eqref{8} is thus
\begin{equation}\label{9aa}
\Ha_I:= \frac{\Ha_E^2 +N(m-1)\Ha_E}{1+N(m-1)}.
\end{equation}
In the following section, we study the complete spectrum of the above operator.

\section{Rigorous spectral analysis}

While the linearization of the confined equation \eqref{8} around the global attractor $v_*$ and thus the definition of displacement Hessian $\Ha_I$ in the previous section were purely formal, our analysis will be rigorous from here on.

\medskip

Let $\H$ denote the class of all locally integrable functions on $\spt(v_*)$ such that
\[
\|\psi\|_{\H}^2:=\int v_* |\grad \psi|^2\, dx\;<\;\infty,
\]
with the identification of functions that only differ by an additive constant. We recall that in the case $m>1$ the Barenblatt profile $v_*$ is compactly supported on a Ball of radius one, while for $m=1$ the Barenblatt profile is a Gaussian and thus positive everywhere, see \eqref{10b}. In any of these cases, the weighted Sobolev space $\H$ is a separable Hilbert space with respect to the topology induced by $\|\cdot\|_{\H}$.

\medskip

The calculations that led to the definition of the displacement Hessians $\Ha_E$ and $\Ha_I$ are certainly valid for functions that are smooth and bounded on the support of the Barenblatt profile, i.e., for $C^{\infty}_b(\spt(v_*))$-functions, and thus, $\Ha_E$ and $\Ha_I$ are both well-defined on $C^{\infty}_b(\spt(v_*))$. (Notice that $C^{\infty}_b(\spt(v_*))$ equals $C_b^{\infty}(\R^N)$ if $m=1$ and $C^{\infty}(\bar B_1)$ if $m>1$.)  As $C^{\infty}_b(\spt(v_*))$ is a dense subspace of $\H$, see, e.g., Lemma \ref{L1} in the appendix or \cite[Lemma 2]{Seis13}, both $\Ha_E$ and $\Ha_I$ are densely defined operators on $\H$. Moreover, they are nonnegative and symmetric, cf.\ proof of Proposition \ref{P1} below, and thus closable in $\H$. Of course, the closures of these operators, still denoted by $\Ha_E$ and $\Ha_I$, are still nonnegative and symmetric.

\medskip

The (closed) operator $\Ha_E$ is self-adjoint with domain
\[
\D(\Ha_E)\;=\;\left\{\psi\in H_{\loc}^3(\spt(v_*)):\: \psi,\, \Ha_E\psi\in\H\right\}.
\]
This feature is presumably 
well-known in the case $m=1$ where $\Ha_E$ is the Ornstein--Uhlenbeck operator and when the underlying Hilbert space is the Gauss space $L^2(e^{-|x|^2/2}dx) = L^2(v_*dx)$. We establish self-adjointness in the Gauss--Sobolev space $\H$ in Proposition \ref{P3} of the appendix. Self-adjointness in the case $m>1$ was recently proved in \cite{Seis13}. Notice that for every $\psi\in\D(\Ha_E)$ we have the integration by parts formula
\begin{equation}\label{9b}
\int v_* \grad\xi\cdot\grad\psi\, dx\;=\; -\int \xi \div\left(v_*\grad\psi\right)\, dx\quad\mbox{for all }\xi \in\H,
\end{equation}
which can be easily seen via approximation with $C_b^{\infty}(\spt v_*)$ functions. In the case $m>1$, this implies the asymptotic boundary condition $v_*\grad\psi\cdot \nu=0$ on $\partial B_1$ which in the case $\psi\in C^{1}( B_1)$ simply becomes
\[
\lim_{|x|\uparrow 1}v_*(x)\grad\psi(x)\cdot\frac{x}{|x|}=0,
\]
cf.\ \cite[Remark 1]{Seis13}. We finally introduce the space
\[
\D(\Ha_I)\;=\; \left\{\psi\in H_{\loc}^5(\spt(v_*)):\: \psi,\,\Ha_E\psi,\, \Ha_E^2\psi\in\H\right\}.
\]
We will see in Proposition \ref{P1} below that $\Ha_I:\D(\Ha_I)\to \H$ is a self-adjoint operator.

\medskip

\subsection{Main results}\label{S:Main}

We are now in the position to state our main results. Let $\N_0=\{0,1,2,\ldots\}$ denote the set of nonnegative integers.

\begin{theorem}\label{T1}
The operator $\Ha_I: \D(\Ha_I)\to \H$ is self-adjoint. Its spectrum is purely discrete and given by the eigenvalues
\[
\mu_{\ell k}:=\frac{\lambda_{\ell k}^2 + N(m-1)\lambda_{\ell k}}{1+N(m-1)},
\]
for $(\ell,k)\in\N_0\times\N_0\setminus\{(0,0)\}$ if $N\ge 2$ and $(\ell,k)\in\{0,1\}\times\N_0\setminus\{(0,0)\}$ if $N=1$, where
\[
\lambda_{\ell k} \;=\;\ell + 2k + 2k(k + \ell + \frac N 2-1)(m-1).
\]
In the case $m=1$, the corresponding eigenfunctions are Hermite polynomials, and in the case $m>1$, the eigenfunctions are given by polynomials of the form
\[
\psi_{\ell n k}(x) =   F(-k, \frac1{m-1} +\ell +\frac{N}2 -1 +k; \ell+\frac{N}2;|x|^2)Y_{\ell n}\left(\frac{x}{|x|}\right)|x|^{\ell},
\]
$n\in\{1,\dots,N_{\ell}\}$ with $N_{\ell}=1$ if $\ell=0$ or $\ell=N=1$ and $N_{\ell}= \frac{(N+\ell-3)!(N+2\ell-2)}{\ell!(N-2)!}$ else, where $F(a,b;c;z)$ is a hypergeometric function and $Y_{\ell n}$ is a spherical harmonic if $N\ge2$, corresponding to the eigenvalue $\ell(\ell + N-2)$ of $-\laplace_{\S^{N-1}}$ with multiplicity $N_{\ell}$. Otherwise, if $N=1$ it is $Y_{\ell n}(\pm1)=(\pm1)^{\ell}$.
\end{theorem}

The values $\lambda_{\ell k}$ are the eigenvalues of the porous medium displacement Hessian operator and both $\Ha_I$ and $\Ha_E$ have the same eigenfunctions, see \cite[Theorem 1]{Seis13} or Theorem \ref{T2} of the appendix. Hence, as expected, the eigenvalues of $\Ha_E$ translate into the eigenvalues of $\Ha_I$ precisely in the same way as the operator $\Ha_E$ translates into the operator $\Ha_I$, cf.\ \eqref{9aa}.

\medskip

Hermite polynomials can be recursively defined by
\[
\psi_0(z)=1, \quad \psi_{n+1}(z) = z\psi_n(z) - \psi_n'(z)\quad\mbox{for }n\in\N,
\]
where $z\in\R$, and then in higher dimensions for every multi-index $\alpha=(\alpha_1,\dots,\alpha_N)\in \N_0^N$ via $\psi_{\alpha}(x) = \psi_{\alpha_1}(x_1)\cdots\psi_{\alpha_N}(x_N)$, cf., e.g., \cite{Thangavelu93,Sjogren97}. Hypergeometric functions $F(a,b;c,z)$ are power series of the form
\begin{equation}\label{9c}
F(a,b;c;z)=\sum_{j=0}^{\infty}\frac{(a)_j(b)_j}{(c)_j j!}z^j,
\end{equation}
where $a,b,c,z\in\R$ and $c$ is not a non-positive integer. The definition involves the Pochhammer symbols
\[
(s)_j=s(s+1)\cdots(s+j-1), \quad\mbox{for }j\ge 1, \quad\mbox{and}\quad(s)_0=1.
\]
Since the hypergeometric functions reduce to polynomials of degree $k$ in the case $a=-k$, the eigenfunctions $\psi_{\ell n k}$ are polynomials of degree $\ell +2k$ and are harmonic if $k=0$. The literature on hypergeometric functions and spherical harmonics is vast, see, e.g., \cite{Abramowitz,BealsWong10,Groemer96}.

\medskip

The knowledge of the complete spectrum of the displacement Hessian operator is a promising starting situation for a full asymptotic expansion of solutions to the fourth-order equation \eqref{1} around the self-similar solution: In view of the discreteness of the spectrum, all modes are in principle accessible. A technical difficulty lies in the incompatibility of the space in which the eigenfunctions live and the space in which solutions  depend differentially on the initial data, and thus, in which the equation can be rigorously linearized. A framework for this differentiable dependency is provided in {\sc Koch}'s habilitation thesis \cite{Koch} in the context of the porous medium equation, recently extended to the fast diffusion regime by {\sc Denzler, Koch}, and the first author \cite{DenzlerKochMcCann13}. In the latter work, the authors study the higher-order asymptotics of solutions to the fast diffusion equation around the self-similar solution based on the spectral analysis of {\sc Denzler} and the first author. However, their expansion is naturally limited to a finite number of modes due to the occurrence of continuous spectrum.  The discreteness of the spectrum of the fourth-order equation (and equivalently of the porous medium equation \cite{Seis13}) is analytically related to the fact that the Barenblatt solutions possess moments to all orders only if $m\ge1$.

\medskip

The eigenvalues of $\Ha_I$ are nonlinear functions of $m$ and eigenvalue crossings occur throughout the spectrum. Moreover, the eigenvalues are increasing functions of the eigenvalues of the porous medium equation $\lambda_{\ell k}$, and thus, we expect the same ordering of eigenmodes with respect to the rate of convergence for both equations. For any value of $N$ and $m$, the dynamics are translation-governed as the smallest eigenvalue $\mu_{10}=1$ corresponds to a spatial translation in direction of the $n$-th coordinate axis $e_n$, for $n\in\{1,\dots,N_1=N\}$. The role of the eigenfunctions is best understood by considering
geodesics in the Wasserstein space, which are given by displacement interpolants (cf.\ \cite{McCann97}) via $v_*(x) = \det\left(I + s\grad^2\psi(x)\right)v_s(x+s\grad\psi(x))$, that is, push-forwards of $v_*$ under the map $\mbox{id} + s\grad\psi$. Observe that $v_s$ generates tangent fields in the Wasserstein gradient flow interpretation of the dynamics, cf.\ Section \ref{S:GF}, since $\left.\partial_s\right|_{s=0} v_s  = -\div\left(v_* \grad \psi\right)$. The eigenfunctions $\psi_{1n0}$ are affine functions and perturbations are thus generated by translations $x + sc_ne_n$ for some constant $c_n\in\R$. A fully rigorous justification of the translation-governed dynamics was obtained by {\sc Carrillo \& Toscani} \cite{CarrilloToscani02} for $N=1$ and $m=3/2$ and by {\sc Matthes, Savar\'e}, and the first author \cite{MatthesMcCannSavare09} for general $N$ and $1\le m\le 3/2$ who prove
\[
d_2(v(t),v_*)\lesssim e^{-t},
\]
where $d$ denotes the Wasserstein distance. This bound is sharp for translations. The second smallest eigenvalue $\mu_{20}$ corresponds to affine transformation, and is generated by transformations $x+sA_{2n}x$ for some symmetric and trace-free matrix $A_{2n}$. Such transformations were studied in detail by {\sc Denzler} and the first author \cite{DenzlerMcCann08}. After the second eigenvalue, a first level crossing occurs and we have $\mu_{01}\le\mu_{30}$ precisely if $N(m-1)\le1$. The eigenvalues $\mu_{30}$ and $\mu_{01}$ correspond to pear-shaped deformations (with order 3 symmetry) and dilations, respectively. The geometric complexity of the higher modes is increasing with the degree of the polynomials, and so we do not attempt to extend this discussion to larger values of $\mu_{\ell k}$.

\subsection{Self-adjointness}\label{S:SA}

In this subsection, we prove that the unbounded operator $\Ha_I: \D(\Ha_I)\to \H$ is self-adjoint and nonnegative. Both properties ensure that the spectrum of $\Ha_I$ is contained in $[0,\infty)$. On a heuristic level, this insight is not surprising: The operator $\Ha_I$ is defined through the Hessian of an energy at a critical point (in fact, it is related to the Hessian by a similarity transformation). Moreover, the critical point is actually the energy minimizer and thus the Hessian must be nonnegative.

\begin{prop}\label{P1}
The operator $\Ha_I: \D(\Ha_I)\to\H$ is positive-definite, self-adjoint, and its spectrum is purely discrete.
\end{prop}

We remark that self-adjointness is {\em not} a consequence of the self-adjointness of $\Ha_E$ alone, because $\T^*\SS^*\subset (\SS\T)^*$ and $(\T + \SS)^*\subset\T^* +\SS^*$ if $*$ denotes the adjoint operation and equality is in general not true. Instead, we also use the positivity of $\Ha_E$ in the sense that zero is in the resolvent set. For a (brief) survey on unbounded operators and the definition of self-adjointness, we refer to \cite[Ch.\ 13]{Rudin} or \cite{Schmudgen}.

\begin{proof} 
%
%
%
%

Observe first that there is no loss of generality in setting all constants to one, so that $\Ha_I=\Ha_E^2 + \Ha_E$. 
It is shown in Lemma \ref{L1} of the appendix and \cite[Lemma 2]{Seis13} that $C_b^{\infty}(\spt(v_*))$ is dense in $\H$. Since $C_b^{\infty}(\spt(v_*))\subset \D(\Ha_I)\subset \H$, we see that $\Ha_I$ is densely defined on $\H$. Moreover, by the self-adjointness and nonnegativity of $\Ha_E$, and since $\D(\Ha_E)\subset \D(\Ha_I)$, it immediately follows that $\Ha_I$ is nonnegative and symmetric on $\D(\Ha_I)$. Indeed, for all $\varphi,\, \psi\in \D(\Ha_I)$, we have that
\begin{eqnarray*}
\int v_*\grad\varphi\cdot \grad\left(\Ha_E^2 \psi \right)\ dx 
&=& m \int v_*^{m-2} \div\left(v_*\grad\varphi\right)\cdot \div\left(v_*\grad\Ha_E\psi\right)\,dx\\
&=& \int v_* \grad\left(\Ha_E\varphi\right)\cdot\grad\left(\Ha_E\psi\right)\, dx,
\end{eqnarray*}
where we have used 
\eqref{9b} both in the first and in the second equality together with $\Ha_E^2\psi\in\H$ and $\Ha_E\varphi\in\H$. This calculation shows both symmetry and nonnegativity of $\Ha_E^2$ on $\D(\Ha_I)$. Since $\Ha_E$ is also non-negative and symmetric on $\D(\Ha_E)\supset\D(\Ha_I)$, the same property holds true for $\Ha_I$.

\medskip

From the fact that $\Ha_I$ is a densely defined, symmetric operator, we deduce that $\Ha_I\subset \Ha_I^*$, where the inclusion has to be understood as an inclusion of the corresponding graphs. We now show that $\Ha_I$ is self-adjoint, i.e., $\Ha_I=\Ha_I^*$. For this, it is enough to show that $\Ha_I$ has an everywhere defined bounded inverse $\Ha_I:\H\to\H$, cf.\ \cite[Theorem 13.11]{Rudin}. Since the spectrum of $\Ha_E$ is contained in $(0,\infty)$, we know that $0$ and $-1$ are both in the resolvent set of $\Ha_E$ and thus both $\Ha_E$ and $\Ha_E + 1$ have a bounded and everywhere defined inverse. 
Therefore, also the composition $\Ha_I^{-1} = \Ha_E^{-1}(\Ha_E + 1)^{-1}$ is bounded and well-defined all over $\H$, which shows that $\Ha_I$ is self-adjoint. Moreover, since $\Ha_E^{-1}$ is actually compact, see \cite[Proposition 2]{Seis13} for the case $m>1$ and Proposition \ref{P4} of the appendix for the case $m=1$, it follows that the resolvent $\Ha_I^{-1}$ is compact as a composition of a bounded and a compact operator. The compactness of the resolvent $\Ha_I^{-1}$ immediately yields that the spectrum of $\Ha_I$ is discrete, cf.\ \cite[Prop.\ 2.11]{Schmudgen}.

\end{proof}

\subsection{Computation of the spectrum of $\Ha_I$}\label{S:proof}
We finally turn to the 

\begin{proof}[Proof of Theorem \ref{T1}] We first verify that the eigenvalue problems for $\Ha_I$ and $\Ha_E$ are equivalent in the sense that $\mu_{\ell k}$ is an eigenvalue of $\Ha_I$ if and only if $\lambda_{\ell k}$ is an eigenvalue of $\Ha_E$. In view of the explicit formula \eqref{9aa}, then the corresponding eigenfunctions have to coincide. For this, it is enough to show that every eigenvalue of $\Ha_I$ translates into an eigenvalue of $\Ha_E$. The reverse implication is trivial.
Let $\mu>0$ be an eigenvalue of $\Ha_I$, i.e., there exists a $\psi\in D(\Ha_I)$ such that $\Ha_I\psi=\mu\psi$.  By the definition of $\Ha_I$ in \eqref{9aa}, this eigenvalue problem for $\Ha_I$ can be converted into  an eigenvalue problem for $\Ha_E$, namely
\[
\Ha_E\left(\Ha_E \psi+\eps \psi\right) = \lambda\left(\Ha_E\psi + \eps\psi\right),
\]
where 
\[
\eps= \frac{N(m-1)}2 + \sqrt{\mu(1+N(m-1)) + \left(\frac{N(m-1)}2\right)^2},\quad \lambda = \frac{\mu(1+N(m-1))}{\eps}.
\]
Notice that $-\eps$ is not an eigenvalue of $\Ha_E$ because $\eps>0$. Now, as $0\not=\Ha_E\psi +\eps\psi\in \D(\Ha_E)$ by the definition of $\D(\Ha_I)$, it follows that $\lambda$ is an eigenvalue of $\Ha_E$.

\medskip

In the case $m=1$, the eigenvalues of the Ornstein--Uhlenbeck operator $\Ha_E$ and the corresponding eigenfunctions are computed in Theorem \ref{T2} in the appendix. The eigenvalues and eigenfunctions of $\Ha_E$ in the case $m>1$ were computed in \cite{Seis13}. As we know by Proposition \ref{P1} that the spectrum is purely discrete, this concludes the proof of Theorem \ref{T1}.
\end{proof}

\section*{Appendix: The spectrum of the Ornstein--Uhlenbeck operator in $\H$.}

In this appendix, we compute the spectrum and the corresponding eigenvalues of the Ornstein--Uhlenbeck operator $\Ha_E:\D(\Ha_E)\to \H$ where $m=1$ and thus $\Ha_E\psi(x) = -\laplace \psi(x) + x\cdot\grad\psi(x)$. Spectral properties of the differential operator $-\laplace +x\cdot \grad$ are well-known to the stochastics community because of its role in stochastic processes, and also in the mathematical physics community because $\Ha_E$ is conjugate to the harmonic oscillator. The only difference between the operator studied here and the ``classical'' Ornstein--Uhlenbeck operator is the choice of the underlying Hilbert space. While the standard choice is the Gauss space $L^2(e^{-|x|^2/2} dx)$, we consider its Sobolev variant $\dot H^1(e^{-|x|^2/2} dx) = \H$. We will see, however, that the spectrum of both operators is identical. More precisely, we have the following

\begin{theorem}\label{T2}
The operator $\Ha_E:\D(\Ha)\to \H$ is self-adjoint. Its spectrum $\sigma(\Ha_E)$ is purely discrete and given
\[
\sigma(\Ha_E)=\N.
\]
The corresponding eigenfunctions are Hermite polynomials.
\end{theorem}

For the convenience of the reader, we will sketch the proof of this result in the sequel. It is based on the following two Propositions:

\begin{prop}\label{P3}
The operator $\Ha_E:\D(\Ha_E)\to \H$ is nonnegative , self-adjoint, and has a bounded inverse.
\end{prop}

\begin{prop}\label{P4}
The operator $\Ha_E: \D(\Ha_E)\to \H$ has a purely discrete spectrum.
\end{prop}

\begin{proof}[Proof of Theorem \ref{T2}]
We immediately deduce from Propositions \ref{P3} and \ref{P4} that the spectrum of $\Ha_E$ is a discrete subset of $(0,\infty)$, and thus it is enough to solve the eigenvalue problem for $\Ha_E$. On the one hand, one can easily show that $\H$ embeds continuously into the Gauss space $L^2(e^{-|x|^2/2} dx)  = L^2(v_*dx)$, cf.\ Lemma \ref{L2} below, and thus every eigenvalue must be an eigenvalue of the ``classical'' operator defined on the Hilbert space $L^2(v_*dx)$. It is well-known that the eigenvalues of the Ornstein--Uhlenbeck operator on $L^2(v_*dx)$ are all positive integers and the corresponding eigenfunctions are Hermite polynomials, cf.\ \cite{Sjogren97}. It is easily checked that every polynomial lies in the domain of $\Ha_E$. Therefore, we conclude that Hermite polynomials are eigenfunctions of $\Ha_E$ and thus $\sigma(\Ha_E)=\N$.
\end{proof}

Before turning to the proofs of Propositions \ref{P3} and \ref{P4}, we derive some auxiliary results:

\begin{lemma}\label{L1}
$C_b^{\infty}(\R^N)$ is dense in $\H$.
\end{lemma}

This is a fairly standard result and we therefore only sketch its proof.

\begin{proof}
We first observe that $L^{\infty}(\R^N)\cap \H$ is dense in $\H$, which can be easily seen by considering the truncated functions $\psi_M=\max\{-M,\min\{M,\psi\}\}$ for $M>0$. It holds that
\[
\lim_{M\uparrow\infty}\int v_* |\grad \psi-\grad\psi_M|^2\, dx\;=\;\lim_{M\uparrow\infty} \int_{|\psi|\ge M} v_*|\grad \psi|^2\, dx\;=\; 0
\]
by the dominated convergence theorem. The density of $C_b^{\infty}(\R^N)$ in $\H$ then follows by a standard mollification argument, see, e.g., \cite[Lemma 2]{Seis13}.
\end{proof}

\begin{lemma}\label{L2}
There exists a constant $C>0$ dependent only on the space dimension $N$ such that for all $\psi\in\H$
\begin{equation}\label{A2}
\inf_{c\in\R} \int (1 +|x|^2)v_* (\psi-c)^2\, dx\;\le\;  C\int v_* |\grad\psi|^2\, dx
\end{equation}
holds.
\end{lemma}

\begin{proof}In the following $C>0$ will always denote a universal constant (eventually dependent on $N$) whose value may change from line to line. Thanks to the density of smooth functions provided in the previous lemma, it is enough to prove the statement for $\psi\in C^{\infty}_b(\R^N)$. We first show that
\begin{equation}\label{A1}
\int (1+ |x|^2) e^{-|x|^2/2} \psi^2\, dx\;\le\; C\left(\int e^{-|x|^2/2} \psi^2\, dx + \int e^{-|x|^2/2} |\grad\psi|^2\, dx\right).
\end{equation}
Indeed, because $\div\left(xe^{-|x|^2/2}\right)  = \left(N- |x|^2\right)e^{-|x|^2/2}$, we have that
\begin{eqnarray*}
\int |x|^2 e^{-|x|^2/2} \psi^2\, dx &=&  - \int \div\left(xe^{-|x|^2/2} \right)\psi^2\, dx + N \int e^{-|x|^2/2} \psi^2\, dx \\
&=& 2\int xe^{-|x|^2/2} \psi\cdot\grad\psi\, dx + N  \int e^{-|x|^2/2} \psi^2\, dx,
\end{eqnarray*}
and we have integrated by parts in the second identity. We apply Young's inequality $2ab\le a^2 + b^2$ to deduce
\[
\int |x|^2 e^{-|x|^2/2} \psi^2\, dx\;\le\; C\left(\int e^{-|x|^2/2} \psi^2\, dx + \int e^{-|x|^2/2} |\grad\psi|^2\, dx\right).
\]
From this, \eqref{A1} follows upon adding $\int e^{-|x|^2/2} \psi^2\, dx$ on both sides of the inequality.

\medskip

We now turn to the proof of \eqref{A2}. We prove a slightly stronger statement by choosing $c= \int_{B_R(0)} \psi\, dx$ for some $R>0$ that has to be fixed later. Equivalently, we may assume that
\[
\int_{B_R(0)} \psi\, dx\;=\;0.
\]
Then the Poincar\'e estimate on the ball $B_R(0)$ reads
\[
\int_{B_R(0)} \psi^2\, dx\;\le\;C R^2 \int_{B_R(0)} |\grad\psi|^2\, dx,
\]
and thus
\[
\int_{B_R(0)} e^{-|x|^2/2} \psi^2\, dx\;\le\; \int_{B_R(0)} \psi^2\, dx\;\le\; Ce^{R^2/2}R^2 \int e^{-|x|^2/2}  |\grad\psi|^2\, dx.
\]
On the other hand, we also have that
\[
\int_{\R^N\setminus B_R(0)}  e^{-|x|^2/2} \psi^2\, dx\;\le \;\frac{1}{R^2} \int |x|^2 e^{-|x|^2/2} \psi^2\, dx.
\]
Consequently, combining the last two estimates with \eqref{A1} yields
\begin{eqnarray*}
\lefteqn{\int (1+|x|^2)e^{-|x|^2/2} \psi^2\, dx}\\
&\le&C\left(\frac1{R^2}\int |x|^2 e^{-|x|^2/2} \psi^2\, dx + e^{R^2/2}R^2\int e^{-|x|^2/2}|\grad\psi|^2\,dx\right).
\end{eqnarray*} 
Choosing $R$ sufficiently large (uniformly in $\psi$), we see that the first term on the right can be absorbed into the left-hand side of the inequality, which yields the statement of the lemma by the definition of $v_*$.
\end{proof}

\begin{lemma}\label{L3}
The embedding of $\H$ in $L^2(v_*dx)$ is compact.
\end{lemma}

\begin{proof}We deduce the statement of this lemma from the standard Rellich compactness lemma for classical Sobolev functions on bounded domains and from estimate \eqref{A2}. Let $\{\psi_n\}_{n\in\N}$ denote a bounded sequence in $\H$. It is convenient to assume that $\int(1+|x|^2)v_*\psi\, dx=0$, because then \eqref{A2} holds with $c=0$. Since $\H$ is a Hilbert space, there exists a $\psi\in\H$ and a subsequence which converges to $\psi$ weakly in $\H$. By the continuous embedding provided by \eqref{A2}, this weak convergence also holds in $L^2((1+|x|^2)v_*dx)$. Moreover, since $v_*$ is bounded away from zero on every compact subset of $\R^N$, it holds that $\{\psi_n\}_{n\in\N}$ is bounded in $H^1(B_k(0))$ for every $k\in\N$. Therefore, by the standard Rellich compactness lemma and since $v_*\lesssim1$, we can extract a further subsequence $\{\psi_{n_k}\}_{k\in\N}$ such that
\begin{equation}\label{A3}
\int_{B_k(0)} v_* (\psi-\psi_{n_k})^2\, dx\;\le\; \frac1k
\end{equation}
for every $k\in\N$. We now have
\begin{eqnarray*}
\int v_*(\psi-\psi_{n_k})^2\, dx &=& \int_{B_k(0)} v_*(\psi-\psi_{n_k})^2\, dx + \int_{\R^N\setminus B_k(0)} v_*(\psi-\psi_{n_k})^2\, dx\\
&\stackrel{\eqref{A3}}{\le}& \frac1k + \frac1{k^2} \int (1+|x|^2)  v_*(\psi-\psi_{n_k})^2\, dx.
\end{eqnarray*}
Since the integral on the right-hand side of the above inequality is bounded by the embedding \eqref{A2}, we deduce that
\[
\int v_*(\psi-\psi_{n_k})^2\, dx\;\le\; \frac1k +\frac{C}{k^2}
\]
for some uniform constant $C>0$. We let $k$ converge to infinity to obtain the desired result.
\end{proof}

\begin{lemma}\label{L4}
For every $u\in L^2(v_*^{-1}dx)$ with $\int u\, dx=0$, there exists a unique $\psi\in H^2_{\loc}\cap \H$ such that
\begin{equation}\label{A5}
\int v_*\grad\psi\cdot \grad\varphi\, dx\;=\; \int u\varphi\, dx
\end{equation}
for all $\varphi\in\H$.
\end{lemma}

\begin{proof}
We first observe that \eqref{A5} are the Euler--Lagrange equations for the convex energy functional
\[
\F(\psi) = \frac12\int v_* |\grad\psi|^2\, dx -\int u\psi\, dx
\]
defined for $\psi\in\H$. Existence and uniqueness of a minimizers follows from soft methods based on the continuous embedding $\H\subset L^2(v_*dx)$ established in Lemma \ref{L2}, whose details we omit as they are fairly standard.
\end{proof}

We are now in the position to prove Propositions \ref{P3} and \ref{P4}. As both statements are established very similarly to the analogous statements Propositions 1 and 2 from \cite{Seis13}, we again omit most of the details. We start with the 

\begin{proof}[Proof of Proposition \ref{P3}] The proof of this proposition is very close to the one of \cite[Prop.\ 1]{Seis13}. By Lemma \ref{L1}, the operator $\Ha_E:\D(\Ha_E)\to\H$ is densely defined. A simple integration-by-parts argument shows that $\Ha_E$ is nonnegative and symmetric. Moreover, via Lemmas \ref{L2} and \ref{L4} it can be shown that $\Ha_E$ is onto, which in turn implies that $\Ha_E$ is self-adjoint and has a bounded inverse via arguments from Functional Analysis (see \cite[Theorem 13.11]{Rudin}).
\end{proof}

It remains to provide the

\begin{proof}[Proof of Proposition \ref{P4}]
By Proposition \ref{P3}, $\Ha_E$ is invertible and has a bounded inverse. To prove the discreteness of the spectrum, we have to show that the resolvent $\Ha_E^{-1}: \H\to\H$ is compact, cf.\ \cite[Prop.\ 2.11]{Schmudgen}. This, however, is a consequence of the Rellich compactness established in Lemma \ref{L3}. We omit details and refer the interested reader to \cite[Prop.\ 2]{Seis13} for a similar argument.

\end{proof}

\bibliography{tfe_lit}

\def\cprime{$'$} \def\cprime{$'$}
\begin{thebibliography}{10}

\bibitem{Abramowitz}
{\sc Abramowitz, M., and Stegun, I.~A.}, Eds.
\newblock {\em Handbook of mathematical functions with formulas, graphs, and
  mathematical tables}.
\newblock Dover Publications Inc., New York, 1992.
\newblock Reprint of the 1972 edition.

\bibitem{AGS}
{\sc Ambrosio, L., Gigli, N., and Savar{\'e}, G.}
\newblock {\em Gradient flows in metric spaces and in the space of probability
  measures}, second~ed.
\newblock Lectures in Mathematics ETH Z\"urich. Birkh\"auser Verlag, Basel,
  2008.

\bibitem{LidiaGiacomelli04}
{\sc Ansini, L., and Giacomelli, L.}
\newblock Doubly nonlinear thin-film equations in one space dimension.
\newblock {\em Arch. Ration. Mech. Anal. 173}, 1 (2004), 89--131.

\bibitem{Barenblatt52}
{\sc Barenblatt, G.~I.}
\newblock On self-similar motions of a compressible fluid in a porous medium.
\newblock {\em Akad. Nauk SSSR. Prikl. Mat. Meh. 16\/} (1952), 679--698.

\bibitem{BealsWong10}
{\sc Beals, R., and Wong, R.}
\newblock {\em Special functions}, vol.~126 of {\em Cambridge Studies in
  Advanced Mathematics}.
\newblock Cambridge University Press, Cambridge, 2010.
\newblock A graduate text.

\bibitem{BenamouBrenier00}
{\sc Benamou, J.-D., and Brenier, Y.}
\newblock A computational fluid mechanics solution to the {M}onge-{K}antorovich
  mass transfer problem.
\newblock {\em Numer. Math. 84}, 3 (2000), 375--393.

\bibitem{BernisFriedman90}
{\sc Bernis, F., and Friedman, A.}
\newblock Higher order nonlinear degenerate parabolic equations.
\newblock {\em J. Differential Equations 83}, 1 (1990), 179--206.

\bibitem{BernisPeletierWilliams92}
{\sc Bernis, F., Peletier, L., and Williams, S.}
\newblock Source type solutions of a fourth order nonlinear degenerate
  parabolic equation.
\newblock {\em Nonlinear Analysis 18\/} (1992), 217--234.

\bibitem{BernoffWitelski02}
{\sc Bernoff, A.~J., and Witelski, T.~P.}
\newblock Linear stability of source-type similarity solutions of the thin film
  equation.
\newblock {\em Appl. Math. Lett. 15}, 5 (2002), 599--606.

\bibitem{BertschDalPassoGarckeGrun98}
{\sc Bertsch, M., Dal~Passo, R., Garcke, H., and Gr{\"u}n, G.}
\newblock The thin viscous flow equation in higher space dimensions.
\newblock {\em Adv. Differential Equations 3}, 3 (1998), 417--440.

\bibitem{CarrilloToscani00}
{\sc Carrillo, J.~A., and Toscani, G.}
\newblock Asymptotic {$L^1$}-decay of solutions of the porous medium equation
  to self-similarity.
\newblock {\em Indiana Univ. Math. J. 49}, 1 (2000), 113--142.

\bibitem{CarrilloToscani02}
{\sc Carrillo, J.~A., and Toscani, G.}
\newblock Long-time asymptotics for strong solutions of the thin film equation.
\newblock {\em Comm. Math. Phys. 225}, 3 (2002), 551--571.

\bibitem{DelPinoDolbeault02}
{\sc Del~Pino, M., and Dolbeault, J.}
\newblock Best constants for {G}agliardo-{N}irenberg inequalities and
  applications to nonlinear diffusions.
\newblock {\em J. Math. Pures Appl. (9) 81}, 9 (2002), 847--875.

\bibitem{DenzlerKochMcCann13}
{\sc Denzler, J., Koch, H., and McCann, R.~J.}
\newblock Higher order time asymptotics of fast diffusion in {E}uclidean space
  (via dynamical systems methods).
\newblock {\em To appear in Mem. Amer. Math. Soc.\/}.
\newblock Preprint arXiv:1204.6434.

\bibitem{DenzlerMcCann05}
{\sc Denzler, J., and McCann, R.~J.}
\newblock Fast diffusion to self-similarity: complete spectrum, long-time
  asymptotics, and numerology.
\newblock {\em Arch. Ration. Mech. Anal. 175}, 3 (2005), 301--342.

\bibitem{DenzlerMcCann08}
{\sc Denzler, J., and McCann, R.~J.}
\newblock Nonlinear diffusion from a delocalized source: affine
  self-similarity, time reversal, \& nonradial focusing geometries.
\newblock {\em Ann. Inst. H. Poincar\'e Anal. Non Lin\'eaire 25}, 5 (2008),
  865--888.

\bibitem{DLSS91a}
{\sc Derrida, B., Lebowitz, J.~L., Speer, E.~R., and Spohn, H.}
\newblock Dynamics of an anchored {T}oom interface.
\newblock {\em J. Phys. A 24}, 20 (1991), 4805--4834.

\bibitem{DLSS91b}
{\sc Derrida, B., Lebowitz, J.~L., Speer, E.~R., and Spohn, H.}
\newblock Fluctuations of a stationary nonequilibrium interface.
\newblock {\em Phys. Rev. Lett. 67}, 2 (1991), 165--168.

\bibitem{FerreiraBernis97}
{\sc Ferreira, R., and Bernis, F.}
\newblock Source-type solutions to thin-film equations in higher dimensions.
\newblock {\em European J. Appl. Math. 8}, 5 (1997), 507--524.

\bibitem{GiacomelliOtto01}
{\sc Giacomelli, L., and Otto, F.}
\newblock Variational formulation for the lubrication approximation of the
  {H}ele-{S}haw flow.
\newblock {\em Calc. Var. Partial Differential Equations 13}, 3 (2001),
  377--403.

\bibitem{GiacomelliOtto03}
{\sc Giacomelli, L., and Otto, F.}
\newblock Rigorous lubrication approximation.
\newblock {\em Interfaces Free Bound. 5}, 4 (2003), 483--529.

\bibitem{GianazzaSavareToscani09}
{\sc Gianazza, U., Savar{\'e}, G., and Toscani, G.}
\newblock The {W}asserstein gradient flow of the {F}isher information and the
  quantum drift-diffusion equation.
\newblock {\em Arch. Ration. Mech. Anal. 194}, 1 (2009), 133--220.

\bibitem{Groemer96}
{\sc Groemer, H.}
\newblock {\em Geometric applications of {F}ourier series and spherical
  harmonics}, vol.~61 of {\em Encyclopedia of Mathematics and its
  Applications}.
\newblock Cambridge University Press, Cambridge, 1996.

\bibitem{Grun04}
{\sc Gr\"un, G.}
\newblock {Droplet spreading under weak slippage --- Existence for the Cauchy
  problem}.
\newblock {\em Comm. Partial Differential Equations 29}, 11-12 (2004),
  1697--1744.

\bibitem{JordanKinderlehrerOtto98}
{\sc Jordan, R., Kinderlehrer, D., and Otto, F.}
\newblock {The variational formulation of the Fokker-Planck equation}.
\newblock {\em SIAM J. Math. Anal. 29}, 1 (1998), 1--17.

\bibitem{JungelPinnau00}
{\sc J\"ungel, A., and Pinnau, R.}
\newblock Global nonnegative solutions of a nonlinear fourth-order parabolic
  equation for quantum systems.
\newblock {\em SIAM J. Math. Anal. 32\/} (2000), 760--777.

\bibitem{Koch}
{\sc Koch, H.}
\newblock {\em Non-{E}uclidean singular integrals and the porous medium
  equation}.
\newblock PhD thesis, Habilitation thesis, Universit\"at Heidelberg, Germany,
  1999.

\bibitem{MatthesMcCannSavare09}
{\sc Matthes, D., McCann, R.~J., and Savar{\'e}, G.}
\newblock A family of nonlinear fourth order equations of gradient flow type.
\newblock {\em Comm. Partial Differential Equations 34}, 10-12 (2009),
  1352--1397.

\bibitem{McCann97}
{\sc McCann, R.~J.}
\newblock A convexity principle for interacting gases.
\newblock {\em Adv. Math. 128}, 1 (1997), 153--179.

\bibitem{Myers98}
{\sc Myers, T.}
\newblock Thin films with high surface tension.
\newblock {\em SIAM Reviews 40\/} (1998), 441--462.

\bibitem{OronDavisBankoff97}
{\sc Oron, A., Davis, S.~H., and Bankoff, S.~G.}
\newblock Long-scale evolution of thin liquid films.
\newblock {\em Rev. Mod. Phys. 69\/} (Jul 1997), 931--980.

\bibitem{Otto98}
{\sc Otto, F.}
\newblock Lubrication approximation with prescribed nonzero contact angle.
\newblock {\em Comm. Partial Differential Equations 23}, 11-12 (1998),
  2077--2164.

\bibitem{Otto01}
{\sc Otto, F.}
\newblock The geometry of dissipative evolution equations: the porous medium
  equation.
\newblock {\em Comm. Partial Differential Equations 26}, 1-2 (2001), 101--174.

\bibitem{Pattle59}
{\sc Pattle, R.~E.}
\newblock Diffusion from an instantaneous point source with a
  concentration-dependent coefficient.
\newblock {\em Quart. J. Mech. Appl. Math. 12\/} (1959), 407--409.

\bibitem{Peterson}
{\sc Petersen, P.}
\newblock {\em Riemannian geometry}, vol.~171 of {\em Graduate Texts in
  Mathematics}.
\newblock Springer-Verlag, New York, 1998.

\bibitem{Rudin}
{\sc Rudin, W.}
\newblock {\em Functional analysis}, second~ed.
\newblock International Series in Pure and Applied Mathematics. McGraw-Hill
  Inc., New York, 1991.

\bibitem{Schmudgen}
{\sc Schm{\"u}dgen, K.}
\newblock {\em Unbounded self-adjoint operators on {H}ilbert space}, vol.~265
  of {\em Graduate Texts in Mathematics}.
\newblock Springer, Dordrecht, 2012.

\bibitem{Seis13}
{\sc Seis, C.}
\newblock Long-time asymptotics for the porous medium equation: {T}he spectrum
  of the linearized operator.
\newblock {\em J. Differential Equations 256}, 3 (2014), 1191--1223.

\bibitem{Sjogren97}
{\sc Sj{\"o}gren, P.}
\newblock Operators associated with the {H}ermite semigroup---a survey.
\newblock In {\em Proceedings of the conference dedicated to {P}rofessor
  {M}iguel de {G}uzm\'an ({E}l {E}scorial, 1996)\/} (1997), vol.~3,
  pp.~813--823.

\bibitem{SmythHill88}
{\sc Smyth, N., and Hill, J.}
\newblock Higher order nonlinear diffusion.
\newblock {\em IMA J. Appl. Math. 40\/} (1988), 73--86.

\bibitem{Thangavelu93}
{\sc Thangavelu, S.}
\newblock {\em Lectures on {H}ermite and {L}aguerre expansions}, vol.~42 of
  {\em Mathematical Notes}.
\newblock Princeton University Press, Princeton, NJ, 1993.
\newblock With a preface by Robert S. Strichartz.

\bibitem{Vazquez03}
{\sc V{\'a}zquez, J.~L.}
\newblock Asymptotic beahviour for the porous medium equation posed in the
  whole space.
\newblock {\em J. Evol. Equ. 3}, 1 (2003), 67--118.
\newblock Dedicated to Philippe B{\'e}nilan.

\bibitem{Vazquez07}
{\sc V{\'a}zquez, J.~L.}
\newblock {\em The porous medium equation}.
\newblock Oxford Mathematical Monographs. The Clarendon Press Oxford University
  Press, Oxford, 2007.
\newblock Mathematical theory.

\bibitem{Villani03}
{\sc Villani, C.}
\newblock {\em Topics in optimal transportation}, vol.~58 of {\em Graduate
  Studies in Mathematics}.
\newblock American Mathematical Society, Providence, RI, 2003.

\bibitem{ZeldovicKompaneec50}
{\sc Zel{\cprime}dovi{\v{c}}, Y.~B., and Kompaneec, A.~S.}
\newblock On the theory of propagation of heat with the heat conductivity
  depending upon the temperature.
\newblock In {\em Collection in honor of the seventieth birthday of academician
  {A}. {F}. {I}offe}. Izdat. Akad. Nauk SSSR, Moscow, 1950, pp.~61--71.

\end{thebibliography}
\bibliographystyle{acm}
\end{document}